\documentclass[12pt,a4paper]{article}
\usepackage[english]{babel}
\usepackage[utf8x]{inputenc}
\usepackage[T1]{fontenc}
\usepackage[a4paper,top=2cm,bottom=2cm,left=1.5cm,right=1.5cm,marginparwidth=1.75cm]{geometry}

\usepackage{amsmath,amsthm,amsfonts}
\usepackage{amssymb}
\usepackage{mathtools}
\usepackage{graphicx}
\usepackage{stmaryrd}
\usepackage[colorinlistoftodos]{todonotes}
\usepackage[colorlinks=true, allcolors=blue]{hyperref}
\usepackage{scalerel}
\usepackage{mathrsfs}
\usepackage{relsize}
\usepackage{enumerate}   
\usepackage{bm,dsfont} 

\newtheorem{theorem}{Theorem}[section]
\newtheorem{lemma}[theorem]{Lemma}
\newtheorem{proposition}[theorem]{Proposition}
\newtheorem{corollary}[theorem]{Corollary}
\theoremstyle{definition}

\theoremstyle{remark}

\newcommand{\C}{\mathbb{C}} 
\newcommand{\R}{\mathbb{R}} 
\newcommand{\N}{\mathbb{N}} 
\newcommand{\Z}{\mathbb{Z}} 
\newcommand{\Q}{\mathbb{Q}} 
\newcommand{\T}{\mathbb{T}} 
\newcommand{\Qp}{\mathbb{Q}_{p}} 
\newcommand{\Zp}{\mathbb{Z}_{p}} 
\newcommand{\AQ}{\mathbb{A}_{\Q}} 
\newcommand{\A}{\mathbb{A}} 
\newcommand{\AQfin}{\mathbb{A}_{\Q,\textnormal{fin}}} 
\newcommand{\Prime}{\mathbb{P}} 
\newcommand{\SO}{\textnormal{\textbf{S}}_{0}} 

\newcommand{\SOprime}{\textnormal{\textbf{S}}_{0}'}

\newcommand{\ghat}{\widehat{G}}

\newcommand{\s}{\textnormal{s}}

\newcommand{\La}{\Lambda}
\newcommand{\la}{\lambda}

\newcommand{\lmodule}{\mathcal{A}}
\newcommand{\rmodule}{\mathcal{B}}

\DeclareBoldMathCommand\boldlangle{\left\langle} 
\DeclareBoldMathCommand\boldrangle{\right\rangle}

\newcommand{\lhs}[2]{\prescript{}{\lmodule\!}{\boldlangle #1,#2\boldrangle}}
\newcommand{\rhs}[2]{{\boldlangle #1,#2\boldrangle}_{\!\rmodule}}

\title{Time-frequency analysis on the adeles over the rationals}
\author{Ulrik B.\ R.\ Enstad\thanks{University of Oslo, Department of Mathematics, Oslo, Norway, \mbox{E-mail: \protect\url{ubenstad@math.uio.no}}}, Mads S.\ Jakobsen\thanks{Norwegian University of Science and Technology, Department of Mathematical Sciences, Trondheim, Norway, \mbox{E-mail: \protect\url{mads.jakobsen@ntnu.no}; \protect\url{franz.luef@ntnu.no}}}, Franz Luef\footnotemark[2]}

\begin{document}
\date{}
\maketitle

\begin{abstract}
We show that the construction of Gabor frames in $L^{2}(\mathbb{R})$ with generators in $\mathbf{S}_{0}(\mathbb{R})$ and with respect to time-frequency shifts from a rectangular lattice $\alpha\mathbb{Z}\times\beta\mathbb{\Z}$ is equivalent to the construction of certain Gabor frames for $L^{2}$ over the adeles over the rationals and the group $\mathbb{R}\times\mathbb{Q}_{p}$. Furthermore, we detail the connection between the construction of Gabor frames on the adeles and on $\R\times\Qp$ with the construction of certain Heisenberg modules.
\end{abstract}

\section{Introduction}

The theory of Gabor systems and their frame properties is available on any locally compact abelian (LCA) group \cite{jale16-2}. However, the construction of explicit examples of Gabor frames with time-frequency shifts from lattices (discrete and co-compact subgroups) is mostly  restricted to the \emph{elementary} LCA groups $\R$, $\Z$, $\T$ and $\Z/d\Z$, $d\in\N$. While the group of the $p$-adic numbers $\Qp$ does not contain any lattice, the group $\R\times\Qp$ and also the adele group over the rationals, $\AQ$, do contain discrete and co-compact subgroups. This makes these groups eligible for Gabor analysis. 

Other efforts to do time-frequency and time-scale (wavelet) analysis on local fields include \cite{ahshsh18,alevsk10,bebe18,kisk10,liji07,sh16,sh18} and \cite{shsk18}.

We mention that \cite{chgo15} provides a method of constructing Gabor frames on any group with time-frequency shifts from lattices. In case of the $p$-adics or other groups that do not contain lattices other methods of building somewhat structured families of functions with the use of \emph{quasi-lattices} have been suggested in \cite{grst07}.

In this paper we combine established theory on the $p$-adic numbers, the adeles, Gabor analysis and the theory of the Feichtinger algebra and the modulation spaces.
Our main result (Theorem \ref{th:gab-frame-for-adele}) shows that the construction of (dual) Gabor frames for $L^{2}(\R)$ with generators in the Feichtinger algebra $\SO(\R)$ is equivalent to both (I) the construction of certain  (dual) Gabor frames for $L^{2}(\AQ)$ with generators in $\SO(\AQ)$ and (II) the construction of certain (dual) Gabor frames for $L^{2}(\R\times\Qp)$ with generators in $\SO(\R\times\Qp)$.

In Section \ref{sec:padic} we describe the groups $\Qp$ and $\AQ$.
The \emph{Feichtinger algebra} on the $p$-adic groups $\Zp$ and $\Qp$, the adeles $\AQ$, and $\R\times\Qp$ is described in Section \ref{sec:SO}. Finally, in Section \ref{sec:gabor} we state and prove our main result. In Section \ref{sec:projections} we connect the construction of dual Gabor frames in $L^{2}(\AQ)$ and $L^{2}(\R\times\Qp)$ with the construction of idempotent elements in
twisted group $C^*$-algebras of the groups $\AQ$ and $\R \times \Q_p$ as explained in \cite{jalu18}. In this way we answer  the question raised in \cite[Remark 5.8]{lapa13}. This work is related to the investigation of noncommutative solenoids by Latr\'emoli\`ere and Packer \cite{lapa13,lapa15,lapa17} and the work of Larsen and Li \cite{lali12-1}.

We also mention that after submitting the preprint for this paper, we have discovered a simultaneous generalization of the main result to Gabor frames over $L^2(\R \times \Omega_a)$, where $\Omega_a$ denotes the group of $a$-adic integers \cite[p.\ 184]{hani98}, \cite[Section 10]{hero63}. These groups generalize both the $p$-adic numbers and the finite adeles over the rationals, but do not appear often in the literature. An example of an application of these groups in operator algebras can be found in \cite{omland_aadic}.

\section{The $p$-adic numbers and adeles over the rationals}
\label{sec:padic}
In this section we give a brief exposition of the field of $p$-adic numbers and the adele group $\A_\Q$ over the rational numbers.
Given a prime number $p$, the \emph{$p$-adic absolute value} on $\Q$ is defined by
\[ |x|_p = p^{-k} \]
where $x = p^k(a/b)$ and $p$ divides neither $a$ nor $b$. One also sets $|0|_p = 0$. The $p$-adic absolute value satisfies a strengthened version of the triangle inequality (the ultrametric triangle inequality), namely
\begin{equation}
| x + y |_p \leq \max \{ |x|_p, |y|_p \}. \label{eq:ultrametric}
\end{equation}
The completion of $\Q$ with respect to the metric $d_p(x,y) = |x-y|_p$ is a field denoted by $\Q_p$ and its elements are called \emph{$p$-adic numbers}. The topology inherited from the metric makes $\Q_p$ into a locally compact Hausdorff space. In particular the $p$-adic numbers $\Qp$ form a (non-compact) locally compact abelian group with respect to the topology induced by the above metric and under addition. One can show that every $p$-adic number $x$ has a \emph{$p$-adic expansion} of the form
\[ x = \sum_{k=-\infty}^\infty a_k p^k, \]
where $a_k \in \{ 0, \ldots, p-1 \}$ for each $k$ and there exists some $n \in \Z$ such that $a_k =0$ for all $k < n$. The sequence $(a_k)_{k \in \Z}$ in this expansion is unique.

\paragraph{The $p$-adic integers.}
The closed unit ball in $\Q_p$ is denoted by $\Z_p$ and its elements are called \emph{$p$-adic integers}. Because of (\ref{eq:ultrametric}) and the multiplicativity of $| \cdot |_p$, $\Z_p$ is a subring of $\Q_p$. In terms of $p$-adic expansions, a $p$-adic number $x = \sum_{k \in \Z} a_k p^k$ is a $p$-adic integer if and only if $a_k = 0$ for $k < 0$. The map $ \{ 0, \ldots, p-1 \}^{\N} \to \Z_p$ given by $(a_k)_k \mapsto \sum_k a_k p^k$ is a homeomorphism, which shows that $\Z_p$ has the topology of a Cantor set. In particular, $\Z_p$ is a compact subgroup of $\Q_p$. But $\Z_p$ is also open in $\Q_p$. Indeed, if $x \in \Z_p$, then using (\ref{eq:ultrametric}) one shows that the open ball $B_{1/2}(x) = \{ y \in \Q_p : |y-x|_p < 1/2 \}$ is contained in $\Z_p$.


We take the Haar measure $\mu_{\Qp}$ on $\Q_p$ so that $\mu_{\Qp}(\Zp)=1$. The Haar measure on $\Z_p$ is the one on $\Q_p$ restricted to $\Z_p$.

\paragraph{The dual groups of $\Q_p$ and $\Z_p$.}
Denote by $\Z[1/p]$ the subring of $\Q$ consisting of rational numbers of the form $a/p^k$ where $k,a\in \Z$. Then $\Q_p = \Z_p + \Z[1/p]$ and $\Z_p \cap \Z[1/p] = \Z$, so that
\[ \Q_p / \Z_p = \frac{\Z_p + \Z[1/p]}{ \Z_p} \cong \frac{ \Z[1/p] }{ \Z_p \cap \Z[1/p] } = \Z[1/p] / \Z \]
as abelian groups. Denote the quotient map $\Q_p \to \Z[1/p]/\Z$ by $x \mapsto \{ x \}_p$. In terms of $p$-adic expansions, we have $\left\{ \sum_{k \in \Z} a_k p^k \right\}_p = \sum_{k=-\infty}^{-1} a_k p^k + \Z$ (observe that for any $p$-adic number $x$ only finitely many of the $a_{k}$ are nonzero). Every character $\omega \in \widehat{\Q}_p$ is of the form
\[ \omega: \Qp \to \C, \ \omega(x) = e^{2\pi i \{ xy \}_p }, \ x \in \Qp, \]
for some $y \in \Q_p$. In fact, the map $\Q_p \to \widehat{\Q}_p$ given by mapping $y$ to the $\omega$ defined above is a topological isomorphism. Hence the Pontryagin dual of $\Q_p$, $\widehat{\Q}_{p}$, can be identified with $\Q_p$ itself. We will use this particular identification for the rest of the paper. 

Recall that the \emph{annihilator} of a closed subgroup $H$ of a locally compact abelian group $G$ is defined by
\[ H^{\perp} = \{ \omega\in \widehat{G} \, : \, \omega(h) = 1 \ \ \text{for all} \ \ h\in H \, \}.\]
In general we have topological isomorphisms $\widehat{H} \cong \widehat{G}/H^{\perp}$ and $\widehat{G/H} \cong H^{\perp}$. Under the identification of $\widehat{\Q}_p$ with $\Q_p$, the annihilator $\Z_p^{\perp}$ of the $p$-adic integers $\Zp$ is identified with $\Z_p$ itself. Hence the dual group $\widehat{\Z}_p$ is isomorphic to $\widehat{\Q}_p/\Z_p^{\perp} \cong \Q_p / \Z_p \cong \Z[1/p]/\Z$. There is another interesting description of $\widehat{\Z}_p$, namely it is isomorphic to the Prüfer $p$-group $\Z(p^\infty)$ which is the subgroup of $\T$ consisting of all $p^n$-th roots of unity as $n$ ranges over all $n=0, 1,2 \ldots$,
\begin{equation} \label{eq:0407a} \Z(p^{\infty}) = \big\{ z \in \C \, : \, z = e^{2\pi i k p^{-n} } , \ k=1,\ldots, p^{n}, \ n\in\N_{0} \, \big\}.\end{equation}
Observe that any $z\in \Z(p^{\infty})$ can be uniquely written as 
$z = \exp ( 2\pi i \sum_{k=1}^{\infty} a_{k} \, p^{-k} )$ where each $a_{k}\in \{0, 1, \ldots, p-1\}$ and only finitely many $a_{k}$ are non-zero. This allows us to identify every element $z\in \Z(p^{\infty})$ with a unique $p$-adic number, also to be denoted by $z$, given by $z=\sum_{k=1}^{\infty} a_{k} \, p^{-k}$. Every element $z\in \Z(p^{\infty})$ defines a character $\omega_{z} \in \widehat{\Z}_{p}$ via
\begin{equation} \label{eq:0907a} \omega_{z}: \Zp\to\C, \ \omega_{z}(x) = e^{2\pi i \{ x  z\}_{p}}, \ x\in \Zp. \end{equation}
Conversely, every element in $\widehat{\Z}_{p}$ is given in this way for some $z\in \Z(p^{\infty})$.


\paragraph{Restricted products.}
Let $(G_{i})_{i\in I}$ be a collection of locally compact abelian groups for some index set $I$. If for each $i\in I$ the group $H_i$ is a compact and open subgroup of $G_i$, then the \emph{restricted product} of the groups $(G_i)_{i\in I}$ (with respect to the $H_i$) is defined to be
\[ G = \sideset{}{^{H_i}}\prod_{i \in I} G_i = \Big\{ (x_i)_{i \in I} \in \prod_{i\in I}G_{i} \ : \ \text{$x_i \in H_i$ for all but finitely many $i \in I$ } \Big\} .\]
The correct topology in order to turn $G$ into a locally compact group is the one with a basis of open sets of the form $\prod_{i \in I} U_i$, where each $U_i$ is open in $G_i$ and $U_i = H_i$ for all but finitely many $i \in I$. This is called the \emph{restricted product topology}. If each $G_i$ is locally compact and each $H_i$ is compact, then the restricted product is a locally compact group \cite[Proposition 5-1(i)]{rava99}. Moreover, every character $\omega \in \widehat{G}$ is of the form
\[ \omega((x_i)_{i \in I}) = \prod_{i \in I} \omega_i(x_i), \]
where $(\omega_i)_{i \in I}$ is an element of the restricted product $\displaystyle\sideset{}{^{H_i^{\perp}}}\prod_{i \in I} \widehat{G}_{i}$, \cite[Theorem 5-4]{rava99}. This gives us an identification
\[ \bigg(\sideset{}{^{H_i}}{\prod\limits_{i \in I}} G_i\bigg)\!{\widehat{\phantom{\Big\langle L}}} \cong \sideset{}{^{H_i^{\perp}}}\prod_{i \in I} \widehat{G}_{i} .\]

If for each $i\in I$ the Haar measure on $G_{i}$ is $\mu_i$, then the product measure $\mu = \prod_{i \in I} \mu_i$ gives a Haar measure on the restricted product of the $G_i$. Typically the Haar measure on each $G_{i}$ is normalized so that $\mu_i(H_i) = 1$.

\paragraph{The adeles.}
The group of \emph{finite adeles} is the restricted product
\[ \AQfin = \sideset{}{^{\Z_p}}\prod_{p \in \Prime} \Q_p .\]
As $\widehat{\Q}_{p}\cong \Qp$ and $\Zp^{\perp}\cong \Zp$, it follows from the previous paragraph that $\AQfin$ is a self-dual locally compact abelian group. 
The group of \emph{adeles} is 
\[ \AQ = \R \times \AQfin. \]
It is also a self dual locally compact abelian group. We write elements of the adeles as $(x_\infty, (x_p)_p)$ where $x_\infty \in \R$ and $(x_p)_p \in \AQfin$.
The Haar measure on $\AQ$ is the product measure of the Lebesgue measures on $\R$ and the measure on all the $p$-adics $\Qp$ (normalized for each $p$ so that $\mu_{\Qp}(\Zp)=1$).

Every element $y=(y_{\infty},(y_{p})_{p})\in \AQ$ defines a character $\omega_{y}\in \widehat{\A}_{\Q}$ via
\begin{equation} \label{eq:2006a} \omega_{y} : \AQ\to \T, \ x=\big(x_{\infty},(x_{p})_{p}\big) \mapsto e^{2\pi i x_{\infty} y_{\infty}} \cdot \prod_{p\in \Prime} e^{-2\pi i \{ x_{p} y_{p}\}_{p}}.\end{equation}
Moreover, every character on $\AQ$ is given in such a way by some $y\in \AQ$. 
The minus in the exponential for the $p$-adics is not necessary for this identification. It is however required for the following neat fact: we will see in a moment that the rationals $\Q$ can be embedded into $\AQ$ as a discrete and co-compact subgroup. By the identifation of $\AQ$ with $\widehat{\A}_{\Q}$ above one has $\Q^{\perp} =\Q$. The identification in \eqref{eq:2006a} is the same that is used in  \cite[Satz 5.4.2]{de10-3}. In \cite[Section 4.3.7]{mapa05} the identification of $\AQ$ with $\widehat{\A}_{\Q}$ is done with the minus in the exponential for the real part.


\paragraph{Lattices in the adeles.}
A subgroup $H$ of an abelian topological group $G$ is called a \emph{lattice} if it is discrete and the quotient group $G/H$ is compact. A \emph{fundamental domain} for $H$ in $G$ is a Borel set $B \subseteq G$ such that every $x \in G$ can be uniquely written as $x = b+h$ where $b \in B$ and $h \in H$. Fundamental domains always exist for lattices in abelian groups \cite[Lemma 2]{kaku98}. The \emph{size} or \emph{covolume} of a lattice $H$ in $G$, denoted by $s(H)$, is the measure $\mu(B)$ of any fundamental domain $B$ for $H$ in $G$.

Naturally, there is an abundance of lattices in $\R$ (they are all of the form $H=\alpha\Z$, $\alpha\in\R\backslash\{0\}$). The $p$-adic numbers contain \emph{no} lattices (the only co-compact subgroup in $\Qp$ is $\Qp$ itself, and the only discrete subgroup of $\Qp$ is the trivial subgroup $\{0\}$). However, both $\AQ$ and $\R\times\Qp$ do contain lattices.

\begin{proposition}\label{prop:lattice_in_adeles}
For any $\alpha \in \R \setminus \{ 0 \}$ the mapping
\[ \varphi_{\alpha}: \Q\to\AQ, \ \varphi_{\alpha}(q) = \big(\alpha q, (q)_{p} \big)\]
embeds $\Q$ as a lattice into $\AQ$. The set $B_{\alpha} = [0,|\alpha|) \times \prod_{p \in \Prime} \Z_p$
is a fundamental domain for $\varphi_{\alpha}(\Q)$ in $\AQ$ and $s(\varphi_{\alpha}(\Q))=\vert \alpha\vert$. 
Moreover, under the identification of $\AQ$ with $\widehat{\A}_{\Q}$ as in \eqref{eq:2006a}, the group $\varphi_{\alpha}(\Q)^{\perp}$ can be identified with $\varphi_{1/\alpha}(\Q)$.
\end{proposition}
\begin{proof}
As described in, e.g., \cite[Theorem 5-11]{rava99}, \cite[Satz 5.2.1]{de10-3}, and \cite[Section 4.3.7]{mapa05}, $\varphi_{1}(\Q)$ is a lattice in $\AQ$ with fundamental domain $[0,1) \times \prod_{p \in \Prime} \Zp$ and $s(\varphi_{1}(\Q)) = 1$. Moreover, it is also shown that $\varphi_{1}(\Q)^{\perp} = \varphi_{1}(\Q)$. The embedding of $\Q$ into $\AQ$ via $\varphi_{\alpha}$ for $\alpha\ne 1$ corresponds to an application of the topological group automorphism $x \mapsto \alpha \cdot x$, $x\in \R$ on the real component of $\AQ$. It is immediate that the desired results hold.
\end{proof}

\paragraph{Lattices in $\R\times\Qp$.}

In a similar fashion to the construction given
in equation \eqref{eq:2006a}, every $y=(y_{\infty},y_{p})\in \R\times\Qp$ defines a character $\omega_{y}\in \widehat{\R}\times\widehat{\Q}_{p}$ via
\begin{equation} \label{eq:2006b} \omega_{y} : \R\times\Qp, \ x=\big(x_{\infty},x_{p}\big) \mapsto e^{2\pi i (x_{\infty} y_{\infty}-\{ x_{p} y_{p}\}_{p})}.
\end{equation}
One can show that every character in $\widehat{\R}\times\widehat{\Q}_{p}$ is given as in \eqref{eq:2006b} for some $y\in \R\times\Qp$.

As for the adeles, there is an abundance of lattices in $\R\times\Qp$. This is well-known and can be found in, e.g., \cite{lapa15}.
\begin{proposition}
Let $p$ be a prime number. For any $\alpha \in \R \setminus \{ 0 \}$ the mapping
\[ \psi_{\alpha}: \Z[1/p]\to\R\times\Qp, \ \psi_{\alpha}(q) = (\alpha q, q )\]
embeds $\Z[1/p]$ as a lattice into $\R\times\Qp$. The set
$B_{\alpha} = [0,|\alpha|) \times \Z_p$
is a fundamental domain for $\psi_{\alpha}(\Z[1/p])$ in $\R\times\Qp$ and $s(\psi_{\alpha}(\Z[1/p])=\vert \alpha\vert$. 
Moreover, under the identification of $\R\times\Qp$ with $\widehat{\R}\times\widehat{\Q}_{p}$ as in \eqref{eq:2006b}, the group $\psi_{\alpha}(\Z[1/p])^{\perp}$ can be identified with $\psi_{1/\alpha}(\Z[1/p])$.
\end{proposition}
The proof is very similar in nature to that of Proposition \ref{prop:lattice_in_adeles} and it is therefore omitted.

\section{Feichtinger's algebra for $p$-adic groups and the adeles}
\label{sec:SO}

For any locally compact abelian group the \emph{Feichtinger algebra}  $\SO$ \cite{fe81-2,ja19,lo80} (sometimes denoted by $\mathbf{M}^{1}$) is a Banach space of functions that behaves very much like the Schwartz-Bruhat space $\mathscr{S}$ (in fact, $\mathscr{S}$ is a dense subspace of $\SO$, see \cite[Theorem 9]{fe81-2}).  For example, $\SO$ is invariant under the Fourier transform and the Poisson formula holds pointwise. 
In this section we describe $\SO$ for the following concrete groups: the real line $\R$, the $p$-adic integers $\Zp$, the $p$-adic numbers $\Qp$, the group $\R\times\Qp$, the finite adeles $\AQfin$, and the adeles $\AQ$.

The description below of $\SO$ on $\Qp$ and $\AQ$ makes it clear that $\SO$ is a far larger and interesting space of functions than the Schwartz-Bruhat space on these groups. The latter consists ``only'' of the collection of all locally constant functions.

\paragraph{$\SO$ on the real line, $\R$.} The Feichtinger algebra on the real line, $\SO(\R)$, is described in detail in, e.g., \cite{fezi98} and \cite{gr01}. Here we only mention the following definition. If we let $g$ be the Gaussian function $g(x) = e^{-x^{2}}$, $x\in \R$, then 
\[ \SO(\R) = \Big\{ f\in L^{1}(\R) \, : \, \int_{\R} \Vert f * E_{\omega}g \Vert_{1} \, d\omega < \infty \Big\}.\]
Here $*$ is the usual convolution of functions and $E_{\omega} : f(x) \mapsto e^{2\pi i \omega x} f(x)$, $x\in\R$ is the modulation operator. The $\SO$-norm of $f$ is given by $\Vert f \Vert_{\SO(\R)} = \int_{\R} \Vert f * E_{\omega}g \Vert_{1} \, d\omega$. 
\paragraph{$\SO$ on the $p$-adic integers, $\Zp$.} Since $\Zp$ is a compact group, it follows from, e.g., \cite[Remark~3]{fe81-2} or \cite[Lemma 4.11]{ja19}, that $\SO(\Zp)$ consists exactly of all continuous functions on $\Zp$ that have an absolutely summable Fourier series. With the identification of the Pr\"ufer $p$-group $\Z(p^{\infty})$ as the dual group of $\Z_{p}$ (see \eqref{eq:0407a} and \eqref{eq:0907a}) we can describe $\SO(\Zp)$ as follows:
\[ \SO(\Zp) = \Big\{ f \in \textbf{C}(\Zp) \, : \, f(x) = \!\sum_{z \in \Z(p^{\infty})} \!\!\! c(z) \, e^{2 \pi i \{xz\}_{p}}, \ x\in \Zp  \ \ \text{and where} \ \ c\in \ell^{1}(\Z(p^{\infty}))\Big\}. \]
Moreover, the norm on $\SO(\Zp)$ is defined by $\Vert f \Vert_{\SO(\Zp)} = \Vert c \Vert_{1}$, where $f$ and $c$ are related as above.

\paragraph{$\SO$ on the $p$-adic numbers, $\Qp$.} The $p$-adic numbers $\Qp$ contain the $p$-adic integers $\Z_{p}$ as a compact open subgroup. A set of coset representatives $Y_{p}$ of $\Qp/\Zp$ is given by 
\begin{equation} \label{eq:coset-Qp-mod-Zp} Y_{p} = \{ y \in \Qp \, : \, y = \sum_{j=1}^{\infty} c_{j} p^{-j}, \ c_{j}\in\{0,1,p-1\}, \ c_{j} = 0 \text{ for all but finitely many $j$} \}.\end{equation}

It follows by, e.g., \cite[Lemma 8(ii)]{fe81-2}, \cite[Theorem 7.7]{ja19}, or \cite[\S 2.9]{re89} that $\SO(\Qp)$ consists exactly of all continuous functions $f$ for which their restrictions to each of the cosets of $\Qp/\Zp$, the collection of functions
\[ \big\{ f_{y}:\Zp\to\C, \ x\mapsto f(x+y), \ x\in\Zp \big\}_{y\in Y_{p}}, \]
belongs to $\SO(\Zp)$ and such that $\Vert f \Vert_{\SO(\Qp)} := \sum_{y\in Y_p} \Vert f_{y} \Vert_{\SO(\Zp)} < \infty$.

The characterization of $\SO(\Zp)$ in terms of the functions with absolutely convergent Fourier series allows us to describe $\SO(\Qp)$ as those functions that are exactly of the form
\begin{equation} \label{eq:1806a}  f(x) = \sum_{z\in \Z(p^{\infty})} c(y,z) \, e^{2\pi i \{(x-y)z\}_{p}} \ \ \text{for all} \ x\in y+\Zp \end{equation}
and for all $y\in Y_{p}$, and where $c\in\ell^{1}(Y_{p}\times\Z(p^{\infty}))$. Moreover, the norm on $\SO(\Qp)$ is equivalently defined by $\Vert f \Vert_{\SO(\Qp)} = \Vert c \Vert_{\ell^{1}(Y_{p}\times \Z(p^{\infty}))}$, where $f$ and $c$ are related as in \eqref{eq:1806a}.

\paragraph{$\SO$ on $\R\times\Qp$.} It follows from, e.g., \cite[Theorem 7]{fe81-2} or \cite[Theorem 7.4]{ja19}, that the functions in $\SO(\R\times\Qp)$ are exactly those of the form
\begin{equation} \label{eq:1806b} f = \sum_{j\in\N} f^{(\R)}_{j} \otimes f^{(\Qp)}_{j} \ \ \text{where} \ \ f^{(\R)}_{j}\in\SO(\R),\, f^{(\Qp)}_{j}\in\SO(\Qp)\end{equation}
for all $j\in\N$ and such that $\sum_{j\in\N} \Vert f^{(\R)}_{j} \Vert_{\SO(\R)} \, \Vert  f^{(\Qp)}_{j} \Vert_{\SO(\Qp)} <\infty$.
The norm on $\SO(\R\times\Qp)$ is given by
\[ \Vert f \Vert_{\SO(\R\times\Qp)} = \inf \big\{ \sum_{j\in\N} \Vert f^{(\R)}_{j} \Vert_{\SO(\R)} \, \Vert  f^{(\Qp)}_{j} \Vert_{\SO(\Qp)} \big\},\]
where the functions $f$, $\{f_{j}^{(\R)}\}_{j\in\N}$ and $\{ f_{j}^{(\Qp)}\}_{j\in\N}$ are related as in \eqref{eq:1806b} and the infimum is taken over all possible representations of $f$ as in \eqref{eq:1806b}.

\paragraph{$\SO$ on $\AQfin$.} The finite adeles $\AQfin$ contain $H=\Pi_{p\in\Prime} \Zp$ as a compact open subgroup. A set of coset representatives $Y$ of $\AQfin/H$ is given by
\begin{align*} Y = \big\{ (y_{2},y_{3},\ldots) \in \Pi_{p\in\Prime} Y_{p} \, : \ & \text{where $Y_{p}$ is as in \eqref{eq:coset-Qp-mod-Zp} and}  \\
& \qquad \text{$y_{p} = 0$ for all but finitely many $p$}\big\}. \end{align*}
By \cite[Lemma 8(ii)]{fe81-2}, \cite[Theorem 7.7]{ja19}, or \cite[\S 2.9]{re89} the Banach space $\SO(\AQfin)$ consists exactly of all continuous functions $f$ on $\AQfin$ for which their restriction to each of the cosets of $\AQfin/H$, the functions
\[  \big\{ f_{y} : H \mapsto \C, \ x \mapsto f(x+y), \ x\in H\big\}_{y\in Y},\]
belong to $\SO(H)$ and such that 
\begin{equation} \label{eq:1806cc} \Vert f \Vert_{\SO(\AQfin)} = \sum_{y\in Y} \Vert f_{y} \Vert_{\SO(H)} < \infty.\end{equation}
Here $\SO(H)$ is the Banach space of continuous functions over the compact group $H=\Pi_{p\in\Prime} \Zp$ with absolutely convergent Fourier series. 

Another characterization of $\SO(\AQfin)$ is given in \cite[\S 9.4]{re89} in the following way: any function $f$ in $\SO(\AQfin)$ is exactly of the form
\begin{equation} \label{eq:1806c} f = \sum_{j\in\N} \bigotimes_{p\in\Prime} f^{(\Qp)}_{j} = \sum_{j\in\N} f^{(\Q_{2})}_{j} \otimes f^{(\Q_{3})}_{j} \otimes \ldots, \ \text{where} \ f_{j}^{(\Qp)} \in \SO(\Qp), \ p\in\Prime, j\in\N,\end{equation}
and for each $j\in\N$ \emph{only finitely many} of the functions $f^{(\Q_{p})}_{j}$, $p\in\Prime$ are \emph{not equal} to $\mathds{1}_{\Zp}$, and such that
\[ \sum_{j\in\N} \prod_{p\in\Prime} \Vert f_{j}^{(\Qp)} \Vert_{\SO(\Qp)} < \infty.\]
The norm 
\[ \Vert f \Vert_{\SO(\AQfin)} = \inf\Big\{ \sum_{j\in\N} \prod_{p\in\Prime} \Vert f_{j}^{(\Qp)} \Vert_{\SO(\Qp)} \Big\},\]
where the infimum is taken over all possible representation of $f$ as in \eqref{eq:1806c}, is a norm on $\SO(\AQfin)$ that is equivalent to the norm in \eqref{eq:1806cc}.

\paragraph{$\SO$ on the adele group, $\AQ$.} By definition $\AQ=\R \times \AQfin$. 
It follows from \cite[Theorem 7]{fe81-2} or \cite[Theorem 7.4]{ja19} that a function $f$ belongs to $\SO(\AQ)$ if and only if
\begin{equation} \label{eq:1806f} f = \sum_{j\in\N} \bigotimes_{p\in\{\infty\}\cup\Prime} f^{(\Qp)}_{j} = \sum_{j\in\N} f_{j}^{(\R)} \otimes f^{(\Q_{2})}_{j} \otimes f^{(\Q_{3})}_{j} \otimes \ldots, \end{equation}
($\Q_{\infty} = \R$) where for each $j\in\N$ only finitely many of the functions $f^{(\Q_{p})}_{j}$, $p\in\Prime$ are not equal to $\mathds{1}_{\Zp}$, and such that
\[ \sum_{j\in\N} \prod_{p\in\{\infty\}\cup\Prime} \Vert f_{j}^{(\Qp)} \Vert_{\SO(\Qp)} < \infty.\]
Moreover, the $\SO(\AQ)$-norm is given by
\[ \Vert f \Vert_{\SO(\AQ)} = \inf \Big\{ \sum_{j\in\N} \prod_{p\in\{\infty\}\cup\Prime} \Vert f_{j}^{(\Qp)} \Vert_{\SO(\Qp)}\Big\},\]
where the infimum is taken over all possible representations of $f$ as in \eqref{eq:1806f}.

\section{Gabor frames}
\label{sec:gabor}
In this section we describe how one can construct Gabor frames for $L^{2}(\R\times\Qp)$ and $L^{2}(\AQ)$ from existing Gabor frames for $L^{2}(\R)$.
The theory of Gabor frames is well understood, see e.g., the books \cite{ch16,gr01} and the recent paper \cite{jale16-2} that develops the theory of Gabor frames for general LCA groups. We give a very brief account of Gabor frames for general LCA groups before we state our main result, Theorem \ref{th:gab-frame-for-adele}.

\paragraph{The theory of Gabor frames for general LCA groups.}
For a moment let $G$ be a general locally compact abelian group. We denote the dual group by $\ghat$. The action that an element $\omega\in \ghat$ has on $x\in G$ is written as $\omega(x)$. For any $x\in G$ and $\omega\in \ghat$ we define the \emph{translation} $T_{x}$ and the \emph{modulation} operator $E_{\omega}$ as follows:
\[ T_{x}f(t) = f(t-x), \ E_{\omega}f(t) = \omega(t) f(t), \ t\in G.\]
The translation and modulation operators are unitary operators on $L^{2}(G)$ and isometries on $\SO(G)$. For convenience we define the \emph{time-frequency} shift operator for any $\la = (x,\omega) \in G\times\ghat$ to be
\[ \pi(\la)  = \pi(x,\omega) = E_{\omega} T_{x}.\]

Let $\Lambda$ be a lattice (a discrete and co-compact subgroup) of the time-frequency domain $G\times\ghat$ and let $g$ be a function in $L^{2}(G)$. We let $\langle \, \cdot \, , \, \cdot \, \rangle$ denote the $L^{2}$-inner product with the linearity in the first entry. The collection of functions $\{\pi(\lambda)g \}_{\lambda\in\Lambda} \subset L^{2}(G)$ is a \emph{Gabor system}. Such a system is a \emph{frame} for $L^{2}(G)$ if there exist constants $A,B>0$ such that
\[ A\, \Vert f\Vert_{2}^{2} \le \sum_{\lambda\in\Lambda} \vert \langle f, \pi(\lambda) g \rangle\vert^{2} \le B \, \Vert f \Vert_{2}^{2} \ \ \text{for all} \ \ f\in L^{2}(G).\]
Equivalently, the associated Gabor frame operator 
\[ S_{g,\La} : L^{2}(G)\to L^{2}(G), \ S_{g,\La} f = \sum_{\la\in\La} \langle f, \pi(\la) g\rangle \pi(\la) g\]
is well-defined, linear, bounded and invertible. 
The usefulness of (Gabor) frames lies in the following.
 If $g\in L^{2}(G)$ and the lattice $\Lambda$ in $G\times\ghat$ are such that $\{\pi(\lambda)g\}_{\lambda\in\Lambda}$ is a frame for $L^{2}(G)$, then there exists a (in general non-unique) function $h\in L^{2}(G)$ such that
\begin{equation}\label{eq:frame-rep} f = \sum_{\lambda\in\Lambda} \langle f, \pi(\lambda) g \rangle \, \pi(\lambda) h \ \ \text{for all} \ \ f\in L^{2}(G).\end{equation}
If $g$ and $h$ are such that \eqref{eq:frame-rep} holds, then they are called a dual pair of Gabor frame generators.
For a given Gabor frame $\{\pi(\la) g\}_{\la\in\La}$ the canonical choice of the function $h$ such that \eqref{eq:frame-rep} holds is the \emph{canonical dual generator} $h = S_{g,\Lambda}^{-1} g$. It is a celebrated result of Gabor analysis that if $g\in \SO(G)$ generates a Gabor frame, then the canonical dual generator also belongs to $\SO(G)$ \cite{grle04}.

For our purposes we mention only the following result of Gabor analysis:
\begin{lemma}[{\cite[Theorem 6.1]{jale16-2}}] \label{le:wex-raz} Let $G$ be an LCA group and let $\Lambda$ be a lattice of $G\times\ghat$. Two functions $g,h\in \SO(G)$ are a dual pair of Gabor frame generators for $L^{2}(G)$ with respect to time-frequency shifts from $\Lambda$ (i.e., \eqref{eq:frame-rep} holds) if and only if
\[ \langle h, \pi(\lambda^{\circ}) g \rangle = \s(\Lambda) \, \delta_{\lambda^{\circ},0} \ \ \text{for all} \ \lambda^{\circ}\in \Lambda^{\circ},\]
where $\Lambda^{\circ}$ is the \emph{adjoint lattice} of $\Lambda$,
\[ \Lambda^{\circ} = \{ \lambda^{\circ} \in G\times\ghat \, : \, \pi(\lambda) \, \pi(\lambda^{\circ}) = \pi(\lambda^{\circ}) \, \pi(\lambda) \ \ \text{for all} \ \ \lambda\in \Lambda\}. \]
\end{lemma}

\paragraph{Gabor systems in $L^{2}(\R)$.} Recall that for every $\omega \in \R$ the modulation operator $E_{\omega}$ is given by $E_{\omega}f(t) = e^{2\pi i \omega t} f(t)$. A Gabor system in $L^{2}(\R)$ generated by a function $g\in L^{2}(\R)$ with time-frequency shifts from the lattice $\La = \alpha\Z\times\beta \Z$, $\alpha,\beta>0$ is thus of the form 
\[ \{ \pi(\lambda) g\}_{\la\in\alpha\Z\times\beta\Z} = \{E_{m\beta}T_{n\alpha}g\}_{m,n\in\Z} = \{ t \mapsto e^{2\pi i m\beta t} g(t-n\alpha)  \}_{m,n\in\Z}.\]

Celebrated results in time-frequency analysis include the following: (A) the Gaussian function $g(x) = e^{-\pi x^{2}}$ \cite{ly92,se92-1} and all totally positive functions \cite{grst13,grrost18} generate a Gabor frame for $L^{2}(\R)$ whenever $\alpha$ and $\beta$ are such that $\alpha \beta<1$.
(B) the values of $\alpha,\beta$ and $\gamma$ such that the Gabor system $\{E_{m\beta}T_{n\alpha} \mathds{1}_{[0,\gamma]}\}_{m,n\in\Z}$ is a frame for $L^{2}(\R)$ is much more difficult to describe \cite{ja03,dasu16}.

\paragraph{Gabor systems in $L^{2}(\R\times\Qp)$.}

For every $\omega = (\omega_{\infty},\omega_{p})\in \R\times\Qp$ the modulation operator $E_{\omega} \equiv E_{\omega_{\infty},\omega_{p}}$ on functions $f$ over $\R\times\Qp$ is defined by
\[ E_{\omega} f(t_{\infty},t_{p}) \equiv E_{\omega_{\infty},\omega_{p}} f(t_{\infty},t_{p}) = e^{2\pi i (\omega_{\infty} t_{\infty}-\{ \omega_{p} t_{p}\}_{p})}f(t_{\infty},t_{p}), \ (t_{\infty},t_{p}) \in \R\times\Qp.\]
A Gabor system generated by a function $g\in L^{2}(\R\times\Qp)$ and the lattice 
\[ \La = \psi_{\alpha}(\Z[1/p])\times \psi_{\beta}(\Z[1/p]) = \left\{  (\alpha q, q, \beta r, r) : q,r\in \Z[1/p] \right\}, \ \alpha,\beta>0\] 
is thus of the form
\[ \{ \pi(\lambda) g\}_{\la\in\La} = \big\{ (t_{\infty},t_{p}) \mapsto e^{2\pi i (\beta r t_{\infty}-\{ r t_{p}\}_{p})} g(t_{\infty}-\alpha q,t_{p}-q) \big\}_{q,r\in\Z[1/p]}. \]

\paragraph{Gabor systems in $L^{2}(\AQ)$.}

For every $\omega = (\omega_{\infty},(\omega_{p})_{p})\in \AQ$ the modulation operator $E_{\omega}\equiv E_{\omega_{\infty},(\omega_{p})_{p}}$ on functions $f$ over $\AQ$ is defined by
\[ E_{\omega_{\infty},(\omega_{p})_{p}} f\big(t_{\infty},(t_{p})_{p}\big) = e^{2\pi i \omega_{\infty}t_{\infty}} \prod_{p\in\Prime} e^{-2\pi i \{ \omega_{p} t_{p}\}_{p}}f\big(t_{\infty},(t_{p})_{p}\big), \ (t_{\infty},(t_{p})_{p}) \in \AQ.\]
A Gabor system generated by a function $g\in L^{2}(\AQ)$ and a lattice 
\[ \La = \varphi_{\alpha}(\Q)\times \varphi_{\beta}(\Q) = \big\{ (\alpha q, (q)_{p}, \beta r, (r)_{p}) : q,r\in \Q \big\}, \ \alpha,\beta>0\]
is thus of the form
\[ \{ \pi(\lambda) g\}_{\la\in\La} = \Big\{ \big(t_{\infty},(t_{p})_{p}\big) \mapsto e^{2\pi i \beta r t_{\infty}} \prod_{p\in\Prime} e^{-2\pi i \{ r t_{p}\}_{p}} g(t_{\infty}-\alpha q,(t_{p}-q)_{p}) \Big\}_{q,r\in\Q}. \]

\paragraph{Gabor frames in $L^{2}(\R\times\Qp)$ and $L^{2}(\AQ)$.}
The following result describes that the construction of a Gabor frame in $L^{2}(\R)$ implies that certain functions generate Gabor frames for $L^{2}(\R\times\Qp)$ and $L^{2}(\AQ)$.

\begin{theorem} \label{th:gab-frame-for-adele} Let $\alpha,\beta>0$. For any two functions $g^{(\R)}$ and $h^{(\R)}$ in $\SO(\R)$
the following statements are equivalent. 
\begin{enumerate}[(i)]
\item $g^{(\R)}$ and $h^{(\R)}$ generate dual Gabor frames for $L^{2}(\R)$ with respect to time-frequency shifts from the lattice $\alpha\Z\times\beta\Z$.
\item For any $p\in \Prime$ the two functions $g=g^{(\R)}\otimes \mathds{1}_{\Zp}$ and $h=h^{(\R)}\otimes \mathds{1}_{\Zp}$ in $\SO(\R\times\Qp)$ generate dual Gabor frames for $L^{2}(\R\times\Qp)$ with respect to the lattice 
\[ \La = \psi_{\alpha}(\Z[1/p])\times \psi_{\beta}(\Z[1/p])= \big\{ (\alpha q, q, \beta r, r) : q,r\in \Z[1/p] \big\} \subset \R\times\Qp\times\R\times\Qp.\]
\item The two functions $g$ and $h$ in $\SO(\AQ)$, defined by
\[g = g^{(\R)}\otimes \mathds{1}_{\Z_{2}}\otimes \mathds{1}_{\Z_{3}} \otimes \ldots \ \  \text{and} \ \  h=h^{(\R)}\otimes \mathds{1}_{\Z_{2}}\otimes\mathds{1}_{\Z_{3}}\otimes\ldots,\] 
generate dual Gabor frames for $L^{2}(\AQ)$ with respect to the lattice 
\[ \Lambda = \varphi_{\alpha}(\Q)\times \varphi_{\beta}(\Q) = \left\{ \big(\alpha q, (q)_{p}, \beta r, (r)_{p}\big) : q,r\in \Q \right\} \subset \R\times\AQfin\times\R\times\AQfin.\]
\end{enumerate}
\end{theorem}
\begin{corollary} \label{cor:gab-frame-for-adele}
For any $g^{(\R)}\in\SO(\R)$ and $\alpha,\beta>0$ the following statements are equivalent. \begin{enumerate}[(i)]
\item The function $g^{(\R)}$ generates a Gabor frame for $L^{2}(\R)$ with respect to the lattice $\alpha\Z\times\beta\Z$.
\item For any $p\in \Prime$ the function $g=g^{(\R)}\otimes \mathds{1}_{\Zp}$ generates a Gabor frame for $L^{2}(\R\times\Qp)$ with respect to the lattice 
\[ \La = \psi_{\alpha}(\Z[1/p])\times \psi_{\beta}(\Z[1/p])= \{ (\alpha q, q, \beta r, r) : q,r \in \Z[1/p] \} \subset \R\times\Qp\times\R\times\Qp.\]
\item The function $g\in \SO(\AQ)$ defined by
\[g = g^{(\R)}\otimes \mathds{1}_{\Z_{2}}\otimes \mathds{1}_{\Z_{3}} \otimes \ldots ,\] 
generates a Gabor frame for $L^{2}(\AQ)$ with respect to the lattice 
\[ \Lambda = \varphi_{\alpha}(\Q)\times \varphi_{\beta}(\Q) = \big\{ \big(\alpha q, (q)_{p}, \beta r, (r)_{p}\big) : q,r\in \Q \big\} \subset \R\times\AQfin\times\R\times\AQfin.\]
\end{enumerate}
\end{corollary}
\begin{proof}[Proof of Theorem \ref{th:gab-frame-for-adele}]
It follows from the description of $\SO$ in Section \ref{sec:SO} that the functions $g$ and $h$ in (ii) and (iii) belong to $\SO(\R\times\Qp)$ and $\SO(\AQ)$, respectively. We only proof the equivalence between (i) and (iii) as the proof of the equivalence between (i) and (ii) is almost identical. Observe that $s(\La) = \alpha\beta$ and that the adjoint lattice to the lattice $\Lambda$ in (iii) is the discrete and co-compact subgroup of $\AQ\times\AQ$ given by
\[ \Lambda^{\circ} = \varphi_{1/\beta}(\Q)\times \varphi_{1/\alpha}(\Q) = \big\{ \big(\beta^{-1} q, (q)_{p}, \alpha^{-1} r, (r)_{p}\big) : q,r\in \Q \big\} \subset \R\times\AQfin\times\R\times\AQfin.\]
By Lemma \ref{le:wex-raz} the two functions $g$ and $h$ generate dual Gabor frames for $L^{2}(\AQ)$ if and only if they satisfy
\begin{equation} \label{eq:0507a} \langle h, 
\pi(\lambda^{\circ})g \rangle = \alpha \beta \, \delta_{\lambda^{\circ},0} \ \ \text{for all} \ \ \lambda^{\circ}\in \La^{\circ}.\end{equation}
The tensor product form of $g$ and $h$ implies that \eqref{eq:0507a} takes the form
\begin{equation} \label{eq:0507b} \langle h^{(\R)}, E_{\alpha^{-1}r} T_{\beta^{-1}q} g^{(\R)}\rangle \, \prod_{p\in\Prime} \langle \mathds{1}_{\Zp}, E_{r}T_{q} \mathds{1}_{\Zp}\rangle = \alpha\beta \, \delta_{(q,r),(0,0)} \ \ \text{for all} \ \ q,r\in \Q.\end{equation}
Observe that a fraction $q\in \Q$ belongs to $\Zp$ for all $p\in \Prime$ if and only if $q\in \Z$. Hence, if $q\in\Q\backslash\Z$, then, for some $p\in \Prime$ the support of the function $\mathds{1}_{\Zp}$ and $T_{q}\mathds{1}_{\Zp}$ is disjoint. 
This implies that \eqref{eq:0507b} is satisfied for all $q\in \Q\backslash\Z$. 
Since modulation $E_{r}$ is turned into translation $T_{r}$ by the Fourier transform on $\Qp$, it follows from Parsevals identity that also, for any given $r\in \Q\backslash\Z$ the inner-product 
\[ \langle \mathds{1}_{\Zp} , E_{r} \mathds{1}_{\Zp} \rangle_{L^{2}(\Qp)} = 0\]
for some $p\in \Prime$.
These two observations imply that we have verified \eqref{eq:0507b} for all $q,r\in \Q\backslash\Z$. It remains to show that \eqref{eq:0507b} holds for all $q,r\in\Z$.
If $q,r\in\Z$, then $\langle \mathds{1}_{\Zp}, E_{r}T_{q} \mathds{1}_{\Zp}\rangle = 1$
for all $p\in\Prime$ (we have normalized the Haar measure on each $\Qp$ such that $\int_{\Qp} \mathds{1}_{\Zp} = 1$). This implies that we only need to verify 
\begin{equation} \label{eq:0507d} \langle h^{(\R)}, E_{\alpha^{-1}r} T_{\beta^{-1}q} g^{(\R)}\rangle = \alpha\beta \delta_{(q,r),(0,0)} \ \ \text{for all} \ \ q,r\in\Z.\end{equation}
Lemma \ref{le:wex-raz} states that 
\eqref{eq:0507d} is satisfied if and only if the two Gabor system $\{E_{m\beta}T_{n\alpha}g^{(\R)}\}_{m,n\in\Z}$ and $\{E_{m\beta}T_{n\alpha}h^{(\R)}\}_{m,n\in\Z}$ are dual Gabor frames, which is (i).
\end{proof}

\paragraph{A Balian-Low theorem.}
The classical Balian-Low theorem for the Feichtinger algebra on the real numbers states that if $g \in \SO(\R)$ and $\alpha \beta = 1$, then $g$ cannot generate a Gabor frame for $L^2(\R)$ over the lattice $\alpha \Z \times \beta \Z \subseteq \R \times \widehat{\R}$. It is a natural question to ask whether the following general statement for locally compact abelian groups holds: Suppose $g \in \SO(G)$ and $\Lambda$ is a lattice in $G \times \widehat{G}$ of the form $\Gamma \times \Gamma^{\perp}$ with $\Gamma$ a lattice in $G$. Then $g$ does not generate a Gabor frame for $L^2(G)$ over $\Lambda$? 

It turns out that this does not  hold for general LCA groups $G$ and lattices $\Gamma$, see \cite{gr98}. However, it has been shown to hold for all lattices in second countable compactly generated LCA groups with noncompact component of the identity \cite{kaku98}. The groups $\A_\Q$ and $\R \times \Q_p$ are both second countable with noncompact component of the identity, but not compactly generated, so the result in \cite{kaku98} does not cover these groups. Note that the lattice $\Gamma = \varphi_{\alpha}(\Q)$ in $\AQ$ has annihilator $\Gamma^{\perp} = \varphi_{1/\alpha}(\Q)$. Thus, if $\alpha \beta = 1$, then $\varphi_{\alpha}(\Q) \times \varphi_{\beta}(\Q) \subseteq \AQ \times \AQ$ is a lattice of the form $\Gamma \times \Gamma^{\perp}$ (and similarly for the lattice $\psi_{\alpha}(\Z[1/p]) \times \psi_{\beta}(\Z[1/p]) \subseteq (\R \times \Q_p) \times (\R \times \Q_p)$).

Combining Corollary \ref{cor:gab-frame-for-adele} and the classical Balian-Low theorem for the Feichtinger algebra on $\R$, we obtain the following restricted Balian-Low type theorem for the groups $\AQ$ and $\R \times \Q_p$:

\begin{proposition}
Let $g \in \SO(\R)$, and let $\alpha \beta = 1$. Then the following hold:
\begin{enumerate}
\item[(i)] The function $g \otimes \mathds{1}_{\Z_2} \otimes \mathds{1}_{\Z_3} \otimes \cdots \in \SO(\AQ)$ does not generate a Gabor frame for  $L^2(\AQ)$ over the lattice $\varphi_{\alpha}(\Q) \times \varphi_{\beta}(\Q)$ in $\AQ \times \AQ$.
\item[(ii)] The function $g \otimes \mathds{1}_{\Z_p} \in \SO(\R \times \Q_p)$ does not generate a Gabor frame for $L^2(\R \times \Q_p)$ over the lattice $\psi_{\alpha}(\Z[1/p]) \times \psi_{\beta}(\Z[1/p])$ in $(\R \times \Q_p) \times (\R \times \Q_p)$.
\end{enumerate}
\end{proposition}

The above result only holds for functions in $\SO(\AQ)$ of the restricted form $g \otimes \mathds{1}_{\Z_2} \otimes \mathds{1}_{\Z_3} \otimes \cdots$ with $g \in \SO(\R)$ (and analogously for $\SO(\R \times \Q_p)$). An interesting question is whether one can obtain the same conclusion for all elements of $\SO(\AQ)$, and the first author is presently working on this.

\paragraph{Modulation spaces.}
Modulation spaces were invented by Feichtinger in the early 80s and can be defined on any locally compact abelian group, see e.g., \cite{fe83-4, fe03-1,fe06} and \cite{gr01}.

It is well-known that the modulation spaces can be described using Gabor frames that are constructed with windows in $\SO$ \cite{gr01}. This characterization of the modulation spaces and the construction of the Gabor frames for $L^{2}(\R\times\Qp)$ and $L^{2}(\AQ)$ in Theorem \ref{th:gab-frame-for-adele} lead to the following.
\begin{lemma} Let $g\in \SO(\AQ)$ and $\La$ be as in Theorem \ref{th:gab-frame-for-adele}(iii) such that $\{\pi(\lambda) g\}_{\lambda\in\Lambda}$ is a Gabor frame for $L^{2}(\AQ)$. The \emph{modulation space} $M^{s,t}(\AQ)$, $s,t\in[1,\infty]$ consists exactly of all elements $\sigma\in \SOprime(\AQ)$ such that
\[ \Vert \sigma \Vert_{M^{s,t}(\AQ)} := \bigg( \sum_{r\in\Q} \Big( \sum_{q\in\Q} \big\vert \sigma\big( E_{\beta r, (r)_{p}}T_{\alpha q, (q)_{p}} g\big) \big\vert^{s} \Big)^{t/s} \bigg)^{1/t}, \]
with the obvious modification if $s$ or $t$ equal $\infty$.
\end{lemma}
In a similar way the modulation spaces on $\R\times\Qp$ can be defined using the Gabor frames constructed as in Theorem \ref{th:gab-frame-for-adele}(ii). 

It is well-known that $\SO\cong M^{1,1}$, $L^{2}\cong M^{2,2}$, and that $\SOprime = M^{\infty,\infty}$.

In recent years the modulation spaces have been used successfully as spaces of symbols in the theory of pseudo-differential operators. For example, the space $M^{\infty,1}$ coincides with the \emph{Sj\"ostrand class}. Among many we refer to, e.g., \cite{gr06-6,grst07}.

\subsection{Heisenberg modules}
\label{sec:projections}
As described in \cite{jalu18}, the construction of Gabor frames for $L^{2}(G)$ with time-frequency shifts from a closed subgroup $\La\subset G\times\ghat$ (where $G$ is any locally compact abelian group and $\ghat$ its dual group) is equivalent to the construction of certain projections in the twisted group $C^*$-algebra $C^*(\La, c)$ where $c$ denotes the cocycle coming from the Heisenberg representation \cite{ri88}. We state briefly some of the theory of \cite{jalu18} for the case of the Gabor frames for $L^{2}(\AQ)$ and $L^{2}(\R\times\Qp)$ constructed in Theorem \ref{th:gab-frame-for-adele}. 
This is the first example of a singly-generated Heisenberg module beyond the case of elementary locally compact abelian groups and it is not covered by the recent results in \cite{jalu18}. The equivalence bimodule that we use here is a suitable completion of the Feichtinger algebra on the respective groups.
In contrast, the theory presented in \cite{ri88} and \cite{lapa13,lapa15} use completions of the Schwartz-Bruhat space $\mathcal{S}(G)$ and $C_c(G)$, respectively, to construct equivalence bimodules between twisted group $C^*$-algebras of lattices in $G \times \widehat{G}$.


\paragraph{Heisenberg modules over $\R\times\Qp$.}
For $\alpha,\beta>0$ we define the following two Banach algebras:
\begin{align*} \lmodule & = \big\{ \mathbf{a} \in \mathsf{B}(L^{2}(\R\times\Qp)) \, : \, \mathbf{a} = \sum_{q,r\in\Z[1/p]} a(q,r) E_{\beta r,r}T_{\alpha q,q}, \ a\in \ell^{1}(\Z[1/p]^{2}) \big\}, \\
\rmodule & = \big\{ \mathbf{b} \in \mathsf{B}(L^{2}(\R\times\Qp)) \, : \, \mathbf{b} = \frac{1}{\alpha\beta}\sum_{q,r\in\Z[1/p]} b(q,r) \big( E_{\alpha^{-1} r,r}T_{\beta^{-1} q,q}\big)^{*}, \ b\in \ell^{1}(\Z[1/p]^{2}) \big\}.\end{align*}

Indeed, the norms $\Vert \mathbf{a} \Vert_{\mathcal{A}} = \Vert a \Vert_{1}$, $\Vert \mathbf{b}\Vert_{\rmodule} = \Vert b \Vert_{1}$ (where $\mathbf{a},a,\mathbf{b}$ and $b$ are related as above) turn $\lmodule$ and $\rmodule$ into involutive Banach algebras with respect to composition of operators and where the involution is the $L^{2}$-adjoint. 

Observe that $\lmodule$ and $\rmodule$ are \emph{not} generated by finitely many unitaries as is the case of the noncommutative 2-torus studied in relation to the construction of Gabor frames in $L^{2}(\R)$ \cite{lu11}.

Elements in $\mathcal{A}$ and $\rmodule$ act on functions in $L^{2}(\R\times\Qp)$ from the left and the right, respectively, by
\begin{align*} 
\mathbf{a}\cdot f &:= \sum_{q,r\in\Z[1/p]} a(q,r) E_{\beta r,r}T_{\alpha q,q} f, \ f\in L^{2}(\R\times\Qp), \ \mathbf{a} \in \mathcal{A}, \\
f\cdot \mathbf{b} &:= \frac{1}{\alpha\beta} \sum_{q,r\in\Z[1/p]} b(q,r) \big( E_{\alpha^{-1} r,r}T_{\beta^{-1} q,q}\big)^{*} f, \ f\in L^{2}(\R\times\Qp), \ \mathbf{b} \in \mathcal{B}.
\end{align*}
We define an $\lmodule$- and $\rmodule$-valued inner product in the following way:
\begin{align*} 
& \lhs{\cdot}{\cdot} : \SO(\R\times\Qp) \times \SO(\R\times\Qp) \to \lmodule, \\
& \qquad \lhs{f}{g} = \sum_{q,r\in\Z[1/p]} \langle f, E_{\beta r,r}T_{\alpha q,q}  g\rangle \,  E_{\beta r,r}T_{\alpha q,q}, \\
& \rhs{\cdot}{\cdot} : \SO(\R\times\Qp) \times \SO(\R\times\Qp) \to \rmodule, \\
& \qquad \rhs{f}{g} = \frac{1}{\alpha\beta} \sum_{q,r\in\Z[1/p]} \langle g, \big( E_{\alpha^{-1} r,r}T_{\beta^{-1} q,q}\big)^{*} f\rangle \,  \big( E_{\alpha^{-1} r,r}T_{\beta^{-1} q,q}\big)^{*}. 
\end{align*}
One can show that
\[ \lhs{f}{g}\cdot h = f \cdot \rhs{g}{h} \ \ \text{for all} \ \ f,g,h\in\SO(\R\times\AQ).\]
Denote by $A$ and $B$ the $C^*$-closures of $\mathcal{A}$ and $\mathcal{B}$ inside $\mathsf{B}(L^2(\R \times \Q_p))$, respectively. The actions and algebra-valued inner products defined give $\SO(\R \times \Q_p)$ the structure of a pre-imprimitivity $\mathcal{A}$-$\mathcal{B}$-bimodule. It can thus be completed into an imprimitivity $A$-$B$-bimodule, a \emph{Heisenberg module} in the sense of Rieffel \cite{ri88}, which sets up a Morita equivalence between $A$ and $B$. It is worth noting that in this case, $A$ is a twisted group $C^*$-algebra on the group $\Z[1/p] \times \Z[1/p]$. These have been termed \emph{noncommutative solenoids} by F. Latrémolière and J. Packer and are studied in \cite{lapa13,lapa15,lapa17}, where they prove that $B$ is also a noncommutative solenoid.
\begin{proposition}[{\cite[Theorem 3.14]{jalu18}}] Let $g$, $h$ be two functions in $\SO(\R\times\Qp)$ and consider the lattices
\begin{align*} \Lambda & = \big\{ (\alpha q, q, \beta r, r) : q,r\in\Z[1/p] \big\}\ \ \text{and} \ \ \Lambda^{\circ} = \big\{ (\beta^{-1} q, q, \alpha^{-1} r, r ) : q,r\in\Z[1/p]  \big\} \end{align*}
in $(\R\times\Qp)^{2}$ with $\alpha,\beta>0$ as in Theorem \ref{th:gab-frame-for-adele}. The following statements are equivalent.
\begin{enumerate}[(i)]
\item $f = \lhs{f}{g}\cdot h$ for all $f\in\SO(\R\times\Qp)$.
\item $ \rhs{g}{h}$ is the identity operator on $L^{2}(\R\times\Qp)$.
\item $g$ and $h$ generate dual Gabor frames with respect to $\La$ for $L^{2}(\R\times\Qp)$.
\item $\lhs{g}{h}$ is an idempotent operator from $L^{2}(\R\times\Qp)$ onto $
V := \overline{\textnormal{span}}\big\{ \pi(\lambda^{\circ}) g\big\}_{\lambda^{\circ}\in\La^{\circ}}$.
\item $f = g \cdot \rhs{h}{f}$ for all $f\in \SO(\R\times\Qp)\cap V$.
\end{enumerate}
\end{proposition}
We close with a result on projections in $\lmodule$ which follows from Theorem \ref{th:gab-frame-for-adele} by choosing for $g^{(\R)}$ the Gaussian $g^{(\R)}_0(t)=e^{-\pi t^2}$.
\begin{proposition}\label{prop:proj-ncsol}
Let $g \in \SO(\R\times\Qp)$ be the function defined by $g = g^{(\R)}_0\otimes \mathds{1}_{\Z_{p}}$,
where $g_0^{(\R)}$ is the Gaussian, and consider the lattice
\[ \Lambda = \varphi_{\alpha}(\Z[1/p])\times \varphi_{\beta}(\Z[1/p]) = \big\{ (\alpha q, q, \beta r, r) : q,r\in\Z[1/p] \big\} \subset \R\times\Qp\times\R\times\Qp.\]
Then $\lhs{S_{g,\Lambda}^{-1/2} g}{S_{g,\Lambda}^{-1/2} g}$ is a projection in $\lmodule$ if and only if $\alpha\beta<1$. 
\end{proposition}
\begin{proof}
  As shown in \cite{lu11} the construction of projections of the form $\lhs{g}{g}$ is equivalent to the construction of tight Gabor frames. Since $S_{g,\Lambda}^{-1/2} g$ generates the canonical tight Gabor frame the result follows from the result of Lyubarskii-Seip that $\{\pi(\alpha k,\beta l)g^{(\R)}_0\}_{k,l\in\Z}$ is a Gabor frame if and only if $\alpha\beta<1$.
 
\end{proof}

\paragraph{Heisenberg modules over $\AQ$.}
For $\alpha,\beta>0$ we define the following two Banach algebras:
\begin{align*} \lmodule & = \big\{ \mathbf{a} \in \mathsf{B}(L^{2}(\AQ)) \, : \, \mathbf{a} = \sum_{q,r\in\Q} a(q,r) E_{\beta r,(r)_{p}}T_{\alpha q,(q)_{p}}, \ a\in \ell^{1}(\Q^{2}) \big\}, \\
\rmodule & = \big\{ \mathbf{b} \in \mathsf{B}(L^{2}(\AQ)) \, : \, \mathbf{b} = \frac{1}{\alpha\beta}\sum_{q,r\in\Q} b(q,r) \big( E_{\alpha^{-1} r,(r)_{p}}T_{\beta^{-1} q,(q)_{p}}\big)^{*}, \ b\in \ell^{1}(\Q^{2}) \big\}.\end{align*}

The norms $\Vert \mathbf{a} \Vert_{\mathcal{A}} = \Vert a \Vert_{1}$, $\Vert \mathbf{b}\Vert_{\rmodule} = \Vert b \Vert_{1}$  turn $\lmodule$ and $\rmodule$ into involutive Banach algebras with respect to composition of operators and where the involution is the $L^{2}$-adjoint. 
As in the case of $\R\times\Qp$ described before, $\lmodule$ and $\rmodule$ are not generated by finitely many unitaries.
Elements in $\mathcal{A}$ and $\rmodule$ act on functions in $L^{2}(\AQ)$ from the left and the right, respectively, by
\begin{align*} 
\mathbf{a}\cdot f &:= \sum_{q,r\in\Q} a(q,r) E_{\beta r,(r)_{p}}T_{\alpha q,(q)_{p}} f, \ f\in L^{2}(\AQ), \ \mathbf{a} \in \mathcal{A}, \\
f\cdot \mathbf{b} &:= \frac{1}{\alpha\beta} \sum_{q,r\in\Q} b(q,r) \big( E_{\alpha^{-1} r,(r)_{p}}T_{\beta^{-1} q,(q)_{p}}\big)^{*} f, \ f\in L^{2}(\AQ), \ \mathbf{b} \in \mathcal{B}.
\end{align*}
We define an $\lmodule$- and $\rmodule$-valued inner product in the following way:
\begin{align*} 
& \lhs{\cdot}{\cdot} : \SO(\AQ) \times \SO(\AQ) \to \lmodule, \\
& \qquad \lhs{f}{g} = \sum_{q,r\in\Q} \langle f, E_{\beta r,(r)_{p}}T_{\alpha q,(q)_{p}}  g\rangle \,  E_{\beta r,(r)_{p}}T_{\alpha q,(q)_{p}}, \\
& \rhs{\cdot}{\cdot} : \SO(\AQ) \times \SO(\AQ) \to \rmodule, \\
& \qquad \rhs{f}{g} = \frac{1}{\alpha\beta} \sum_{q,r\in\Q} \langle g, \big( E_{\alpha^{-1} r,(r)_{p}}T_{\beta^{-1} q,(q)_{p}}\big)^{*} f\rangle \,  \big( E_{\alpha^{-1} r,(r)_{p}}T_{\beta^{-1} q,(q)_{p}}\big)^{*}. 
\end{align*}
One can show that
\[ \lhs{f}{g}\cdot h = f \cdot \rhs{g}{h} \ \ \text{for all} \ \ f,g,h\in\SO(\AQ).\]
In this case, as in the case with the group $\R \times \Q_p$, we obtain from $\SO(\AQ)$ an imprimitivity $A$-$B$-bimodule, where $A$ and $B$ denote the $C^*$-closures of $\mathcal{A}$ and $\mathcal{B}$ in $\mathsf{B}(L^2(\AQ))$, respectively.
\begin{proposition}[{\cite[Theorem 3.14]{jalu18}}] Let $g$, $h$ be two functions in $\SO(\AQ)$ and consider the lattices
\begin{align*} \Lambda & = \big\{ ( \alpha q, (q)_{p}, \beta r, (r)_{p} ) : q,r\in\Q \big\} \ \ \text{and} \ \ \Lambda^{\circ} = \big\{( \beta^{-1} q, (q)_{p}, \alpha^{-1} r, (r)_{p} ) : q,r \in \Q \big\} \end{align*}
in $\AQ^{2}$ with $\alpha,\beta>0$ as in Theorem \ref{th:gab-frame-for-adele}. The following statements are equivalent.
\begin{enumerate}[(i)]
\item $f = \lhs{f}{g}\cdot h$ for all $f\in\SO(\AQ)$.
\item $ \rhs{g}{h}$ is the identity operator on $L^{2}(\AQ)$.
\item $g$ and $h$ generate dual Gabor frames with respect to $\La$ for $L^{2}(\AQ)$.
\item $\lhs{g}{h}$ is an idempotent operator from $L^{2}(\AQ)$ onto $
V := \overline{\textnormal{span}}\big\{ \pi(\lambda^{\circ}) g\big\}_{\lambda^{\circ}\in\La^{\circ}}$.
\item $f = g \cdot \rhs{h}{f}$ for all $f\in \SO(\AQ)\cap V$.
\end{enumerate}
\end{proposition}
We close with a result on projections in $\lmodule$ which follows from Theorem \ref{th:gab-frame-for-adele} by choosing for $g^{(\R)}$ the Gaussian $g^{(\R)}_0(t)=e^{-\pi t^2}$. The proof is analogous to the one of Proposition \ref{prop:proj-ncsol}. 
\begin{proposition}
Let $g \in \SO(\AQ)$ be the function defined by
\[g = g^{(\R)}_0\otimes \mathds{1}_{\Z_{2}}\otimes \mathds{1}_{\Z_{3}} \otimes \ldots ,\] 
where $g_0^{(\R)}$ denotes the Gaussian, and consider the lattice
\[ \Lambda = \varphi_{\alpha}(\Q)\times \varphi_{\beta}(\Q) = \big\{ ( \alpha q, (q)_{p}, \beta r, (r)_{p} ) : q,r\in \Q \big\} \subset \R\times\AQfin\times\R\times\AQfin.\]
Then $\lhs{S_{g,\Lambda}^{-1/2} g}{S_{g,\Lambda}^{-1/2} g}$ is a projection in $\lmodule$ if and only if $\alpha\beta<1$. 
\end{proposition}

\subsection*{Acknowledgments}
The work of M.S.J.\ was carried out during the tenure of the ERCIM ``Alain Bensoussan'' Fellowship Programme at NTNU. This work was finished while M.S.J.\ and F.L.\ were visiting NuHAG at the Faculty of Mathematics at the University of Vienna. We are thankful for their hospitality. The first author thanks Nadia Larsen and Tron Omland for valuable discussions about the groups $\AQ$ and $\R \times \Q_p$ and for pointing out that these groups admit lattices.

\bibliographystyle{abbrv}

\end{document}